\documentclass[
11pt]{amsart}

\usepackage{amsmath}
\usepackage{amssymb}
\usepackage{latexsym}
\usepackage{amsfonts}
\usepackage{graphpap}

\begin{document}

\newcommand{\ci}[1]{_{ {}_{\scriptstyle #1}}}
\newcommand{\ti}[1]{_{\scriptstyle \text{\rm #1}}}

\newcommand{\norm}[1]{\ensuremath{\|#1\|}}
\newcommand{\abs}[1]{\ensuremath{\vert#1\vert}}
\newcommand{\nm}{\,\rule[-.6ex]{.13em}{2.3ex}\,}

\newcommand{\lnm}{\left\bracevert}
\newcommand{\rnm}{\right\bracevert}

\newcommand{\p}{\ensuremath{\partial}}
\newcommand{\pr}{\mathcal{P}}

\newcounter{vremennyj}

\newcommand\cond[1]{\setcounter{vremennyj}{\theenumi}\setcounter{enumi}{#1}\labelenumi\setcounter{enumi}{\thevremennyj}}

\newcommand{\pbar}{\ensuremath{\bar{\partial}}}
\newcommand{\db}{\overline\partial}
\newcommand{\D}{\mathbb{D}}
\newcommand{\T}{\mathbb{T}}
\newcommand{\C}{\mathbb{C}}
\newcommand{\N}{\mathbb{N}}
\newcommand{\bP}{\mathbb{P}}

\newcommand{\bS}{\mathbf{S}}
\newcommand{\bk}{\mathbf{k}}

\newcommand\cE{\mathcal{E}}
\newcommand\cP{\mathcal{P}}
\newcommand\cC{\mathcal{C}}
\newcommand\cH{\mathcal{H}}
\newcommand\cU{\mathcal{U}}
\newcommand\cQ{\mathcal{Q}}

\newcommand{\be}{\mathbf{e}}

\newcommand{\la}{\lambda}
\newcommand{\e}{\varepsilon}

\newcommand{\td}{\widetilde\Delta}

\newcommand{\tto}{\!\!\to\!}
\newcommand{\wt}{\widetilde}
\newcommand{\shto}{\raisebox{.3ex}{$\scriptscriptstyle\rightarrow$}\!}

\newcommand{\La}{\langle }
\newcommand{\Ra}{\rangle }
\newcommand{\ran}{\operatorname{ran}}
\newcommand{\tr}{\operatorname{tr}}
\newcommand{\codim}{\operatorname{codim}}
\newcommand\clos{\operatorname{clos}}
\newcommand{\spn}{\operatorname{span}}
\newcommand{\lin}{\operatorname{Lin}}
\newcommand{\rank}{\operatorname{rank}}
\newcommand{\re}{\operatorname{Re}}
\newcommand{\vf}{\varphi}
\newcommand{\f}{\varphi}


\newcommand{\entrylabel}[1]{\mbox{#1}\hfill}

\newenvironment{entry}
{\begin{list}{X}%
  {\renewcommand{\makelabel}{\entrylabel}%
      \setlength{\labelwidth}{55pt}%
      \setlength{\leftmargin}{\labelwidth}
      \addtolength{\leftmargin}{\labelsep}%
   }%
}%
{\end{list}}



\numberwithin{equation}{section}

\newtheorem{thm}{Theorem}[section]
\newtheorem{lm}[thm]{Lemma}
\newtheorem{cor}[thm]{Corollary}
\newtheorem{prop}[thm]{Proposition}

\theoremstyle{remark}
\newtheorem{rem}[thm]{Remark}
\newtheorem*{rem*}{Remark}

\title[Curvature for non-contractions]{Curvature condition for non-contractions does not imply similarity to the backward shift.}
\author{Hyun-Kyoung Kwon and Sergei Treil}  
\thanks{The work of S.~Treil was supported by the National Science Foundation under Grant  DMS-0501065}
\address{Department of Mathematics\\ Brown University\\ 151 Thayer Street Box 1917\\ Providence, RI USA  02912}


\begin{abstract}
We give an example of an operator that satisfies the curvature condition as defined in \cite{KwonTreil}, but is not similar to the backward shift  $S^*$ on the Hardy class $H^2$. We conclude therefore that the contraction assumption in the similarity characterization given in \cite{KwonTreil} is a necessary requirement. \end{abstract}

\maketitle
\setcounter{tocdepth}{1}
\tableofcontents
\section*{Notation}
\begin{entry}
\item[$:=$] equal by definition;\medskip
\item[$\C$] the complex plane;\medskip
\item[$\D$] the unit disk, $\D:=\{z\in\C:\abs{z}<1\}$;\medskip
\item[$\frac{\p}{\p z}, \frac{\p}{\p \overline z}$] $\p$ and $\db$ derivatives: 
$\frac{\p}{\p z}  := (\frac{\p}{\p x} - i \frac{\p}{\p y})/2$, $\frac{\p}{\p \overline z} 
 := (\frac{\p}{\p x} + i \frac{\p}{\p y})/2$, the symbols $\p$ and $\db$ are sometimes  used; \medskip
\item[$\Delta$]normalized Laplacian, $\Delta = \db \p = \p
\db = \frac14\left(\frac{\p^2}{\p x^2} + \frac{\p^2}{\p y^2} \right)$;\medskip
\item[$\norm{\cdot}$]  norm;\medskip
\item[$H^2$, $H^\infty$] Hardy classes of analytic functions,
$$
H^p := \left\{ f\in L^p(\T) : \hat f (k) := \int_\T f(z) z^{-k}
\frac{|dz|}{2\pi} = 0\ \text{for } k<0\right\},
$$
Hardy classes can be identified with the spaces of functions that
are analytic in the unit disk $\D$: in particular, $H^\infty$ is
the space of all functions bounded and
analytic in $\D$. \medskip

\end{entry}

\section{Introduction and result}

In this paper we consider operators with a complete analytic family of eigenvectors (the backward shift in the Hardy space $H^2$, or, more generally in some reproducing kernel Hilbert space being a typical example). We are interested in the classification of such operators in terms of the geometry of their eigenvector bundles. 

The problem of unitary classification was completely solved by M. J. Cowen and R. G. Douglas in \cite{CowenDouglas}. One of the first results in \cite{CowenDouglas} was a version of Calabi's Rigidity Theorem, which stated that if the eigenvector bundles of operators $T_1$ and $T_2$ are isomorphic (as Hermitian vector bundles), then the operators are unitarily equivalent. In the case of $\dim\ker (T_{1,2}-\lambda I) =1$, the fact that the eigenvector bundles are isomorphic simply means that the curvatures coincide. Note, that the general case of $\dim\ker(T_{1,2}- \lambda I) = n<\infty$ was also treated in  \cite{CowenDouglas}, and numerous local criteria for determining unitary equivalence were obtained there. 

The problem of classification of operators up to similarity, as some examples in \cite{CowenDouglas} did show, is significantly more complicated. For example, an obvious conjecture that uniform equivalence of 
curvatures is sufficient for similarity (it is obviously necessary) is trivially wrong; the curvatures of the (eigenvector bundles of the) backward shifts in $H^2$ and in the Bergman space $A^2$ differ by a factor of $2$, but these operators are very far from being similar. 

In \cite{KwonTreil} a particular case of the similarity problem, namely the problem of similarity to the backward shift in the Hardy space $H^2$ was considered.  The structure of the backward shift is  very well understood, so there was a hope that it will be possible to get a solution in this special case, and that this solution may lead to a better understanding of the general case.

This program was partially realized by the authors: in \cite{KwonTreil} a complete description (in terms of the geometry of the eigenvector bundles) of operators similar to backward shifts of finite multiplicity was obtained, but under an additional assumption that the operator $T$ is a contraction, $\|T\|\le 1$. 
%

More precisely, it was assumed in \cite{KwonTreil} that the operator $T$ in a Hilbert space $H$ satisfies the following 4 conditions:

\begin{enumerate}

\item[(i)]
$T$ is contractive, i.e., $\|T\|\leq 1$;
\item[(ii)]
$\dim \ker(T-\lambda I)=n < \infty$ for all $\lambda\in
\D$;

\item[(iii)]
$\spn \{\ker (T-\lambda I):\lambda\in \D\}=H$ ;

\item[(iv)] The subspaces 
$\cE(\la)=\ker(T-\lambda I)$   depend analytically on the spectral parameter
$\la \in \D$.\
\end{enumerate}

The main result of \cite{KwonTreil} was a necessary and sufficient condition for the operator $T$ to be similar to the backward shift of multiplicity $n$, i.e.,~to the direct sum of $n$ backward shifts $S^*$ in $H^2$, $S^*f(z) = (f(z)-f(0))/z$, $f\in H^2$.  For the case $n=1$, which we are considering in this paper, this result can be stated as folllows:

\begin{thm} [Theorem 0.1, \cite{KwonTreil}]
\label{t0.1}
The following statements are equivalent:
\begin{enumerate}
\item
T is similar to the backward shift operator $S^*$ on $H^2$;

\item
The eigenvector bundles of $T$ and $S^*$ are ``uniformly equivalent'', i.e., there exists a holomorphic bundle map bijection $\Psi$ from the eigenvector bundle of
$S^*$ to that of $T$  such that for some constant $c>0$, 
$$
\frac{1}{c} \| v_{\lambda}  \|_{H^2} \leq \| \Psi ( v_{\lambda} )\|
_{H} \leq c \| v_{\lambda} \| _{H^2} 
$$
for all $v_{\lambda} \in \ker(S^* -\lambda I)$ and for all $\la\in \D$; 

\item
There exists a bounded subharmonic solution $\vf$ to the Poisson equation
$$
\Delta \vf (z) =|\kappa_{S^*}(z)-\kappa_T(z)|
$$
for all $z \in \D$,  where $\kappa_{S^*}$ and $\kappa_{T}$ denote the curvatures of the eigenvector bundles of the operators $S^*$ and $T$, respectively.
We say that the ``curvature condition'' is satisfied in this case.

\item 
The measure 
$$
|\kappa_{S^*}(z)-\kappa_T(z)
|(1-|z|)dxdy,
$$ 
where $\kappa_{S^*}$ and $\kappa_{T}$ are the corresponding curvatures,
is  Carleson and the estimate 
$$
{|\kappa_{S^*}(z)-\kappa_T(z)
|}^{\frac{1}{2}} \le
\frac{C}{1-|z|}
$$
holds.

\end{enumerate}

\end{thm}

\begin{rem}
\label{r1.2}
Without stating the exact definition, let us recall that the curvature $\kappa_T$ of the eigenvector bundle of $T$  (in fact, the curvature of any Hermitian line bundle over a planar domain) can be calculated as $\kappa_T(z)=-\Delta \ln \|\gamma(z)\|^2$, where $\gamma$ is a section of the bundle and  $\Delta$ is the normalized Laplacian, i.e., $\Delta = \db \p = \p
\db$ \cite{CowenDouglas}. 
%
%

It can also be expressed as $\kappa_T(z)=-|\frac{\p \Pi(z)}{\p z}|^2$, where $\Pi:\D\rightarrow
B(H)$ denotes the projection-valued function that assigns to each $\lambda \in \D$, the orthogonal
projection onto $\ker (T-\lambda I)$, $
\Pi(\lambda):= P_{\ker(T-\lambda I)}$ (this formula is in fact true for any line subbundle of the trivial bundle $\Omega\times H \to \Omega$, where $H$ is a Hilbert space and $\Omega\subset\C$).

In \cite{KwonTreil}, the latter formula for the curvature was used, but it is well known and not hard to show that both formulas give the same result. 
\end{rem}

\begin{rem}
For the backward shift $S^*$, the curvature $\kappa_{S^*}$ can be easily computed: $\kappa_{S^*}(z) = (1-|z|^2)^{-2}$. 
\end{rem}

\begin{rem}
It was shown in \cite{KwonTreil} that if $T$ is a contraction (with $\dim (T-\lambda I) = 1$ for all $\la\in \D$), then 
 $\kappa_{S^*}-\kappa_T \geq 0$.  So in the statement of the main theorem (Theorem 0.1) in \cite{KwonTreil}, the term $\kappa_{S^*}-\kappa_T$ was used instead of $|\kappa_{S^*}-\kappa_T|$. This nonnegativity  is consistent with a more general result in \cite{Misra}, but we can no longer take it for granted when it comes to noncontractions.
\end{rem}

A natural question is  whether it is possible to get rid of the assumption that $T$ is contractive in Theorem \ref{t0.1}.   This question was partially answered in \cite{KwonTreil}, where a noncontractive operator satisfying  statement (2), but which is not similar to $S^*$ was constructed. That means the assumption $\|T\|\le1$  is needed to prove Theorem \ref{t0.1} in full generality.  

However, there was  hope, that maybe condition (3) and/or (4) is sufficient for similarity even without assuming that $T$ is a contraction. 

In the present paper, we destroy this hope by  giving an example of an operator that satisfies statements (3) and (4), but is not similar to $S^*$. This example is a slight modification of the one given in \cite{KwonTreil}.  

Now we state the main result:
\begin{thm}
\label{t0.4}
Given $\alpha>0$, there exists an operator $T$  that is ``almost'' a coisometry in the sense that $(1-\alpha)\|x\| \leq \|T^*x\| \leq (1+\alpha) \|x\|$, and such that it is not similar to $S^*$ on $H^2$, but satisfies statements (2)-(4) of Theorem 0.1.
\end{thm}
 
\section{The construction}
The operator  $T$ will be  constructed as the backward shift on the weighted Hardy  class  $H^2_w:=\left\{f = \sum_{n\ge0} a_n z^n : \|f\|^2_w:= \sum_{n\ge 0} |a_n|^2 w_n <\infty \right\}$, where $w=\{w_n\}_0^\infty$, $w_n>0$. 
Let us remind the reader that the backward shift in $H^2_w$ is  the adjoint of the forward shift $f\mapsto zf$, and it can be easily shown that 
$$
T\left(\sum_{n=0}^\infty a_n z^n \right) = \sum_{n=0}^\infty \frac{w_{n+1}}{w_n} a_{n+1} z^n. 
$$
Let us also remind the reader that the eigenvector bundle of $T$ is the family of the reproducing kernel $k^w_\la$  of $H^2_w$,
$$
T k_\la^w = \overline\la k_\la^w, \qquad \la\in \D,  
$$
and that 
%
the reproducing kernel $k^w_\la$ can be computed by the following formula:
\begin{equation}
\label{1.1}
k_\la^w(z) = \sum_{k=0}^\infty \frac{\overline \la^n z^n}{w_n}. 
\end{equation}

 The weight sequence $w=\{w_n\}_0^\infty$ ($w_n> 0$) will be of the form (see Fig.~\ref{fig1})
$$
\ln w_n=
\begin{cases} 2j \ln(1+\alpha) &n=N_k+j,\  0 \leq j \leq k,
\\
2j \ln(1+\alpha) &n=N_k+2k-j, \ 0 \leq j \leq k,
\\
0 & \text{ otherwise, }
\end{cases}
$$
where the numbers $N_k$, $k\in\N$, satisfying $N_k+2k < N_{k+1}$ will be chosen later. 

\setlength{\unitlength}{0.6pt}
\begin{figure}
\begin{center}
\begin{picture} (350,300)

\put (-10,32){$0$}
\put (-0, 227){$2k\ln(1+\epsilon)$}
\put (-15,10){\vector(0,1){250}}
\put(-25,50){\vector(1,0){380}}
\put (-25, 270){$\ln w_n$}
\put (365,47){$n$}

\put(30,45){\line(0,1){8}}
\put(50,45){\line(0,1){8}}
\put(70,45){\line(0,1){8}}

\put(135,45){\line(0,1){8}}
\put(195,45){\line(0,1){8}}
\put(255,45){\line(0,1){8}}
\put(125,30){$N_{k}$}
\put(160, 30){$N_{k}+k$}
\put(245,30){$N_{k}+2k$}

\put(-20,120){\line(1,0){10}}
\put(-20,230){\line(1,0){10}}


\thicklines
\put(-15,50){\line(1,0){45}}

\thicklines
\put(30,50){\line(1,3){20}}

\thicklines
\put(50,110){\line(1,-3){20}}

\thicklines
\put(70,50){\line(1,0){65}}

\thicklines
\put(135,50){\line(1, 3){60}}

\thicklines
\put(195,230){\line(1,-3){60}}

\thicklines
\put(255,50){\line(1,0){100}}

\end{picture}
\end{center}
\caption{The function $\ln w_n$: two ``spikes'' are shown}%
{\protect\label{fig1}}%
\end{figure}
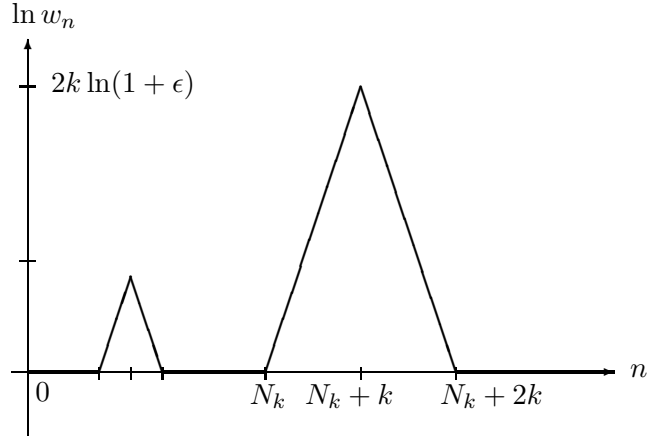

Using the fact that  the weight sequence $w=\{w_n\}_0^\infty$ is unbounded, one can show that the operator $T$ is not similar to $S^*$. To see this, let us assume the contrary, i.e., $AT=S^*A$ for some bounded, invertible operator $A$ so that ${T^*}^nA^*=A^*S^n$. If we take an $f\in H^2$ such that $A^*f=\sum_0^\infty a_n z^n \ne0$ and an $m$ such that  $a_m\ne 0$, then  
$$
\| {T^*}^n A^* f\|^2 = \sum_{j=0}^\infty |a_j|^2 w_{j+n} \ge |a_m|^2 w_{m+n}, 
$$
so $\sup_n \|{T^*}^n A^*f\| = \infty$.  But
$\|A^*S^n f\|_{H^2_w} \leq \|A^*\|\|f\|_{H^2}$,
and we have  a contradiction.

Also, it is easy to see that the adjoint $T^*$ of the operator is ``almost'' an isometry, i.e.,
$$
(1-\alpha)\|x\| \leq \|T^*x\| \leq (1+\alpha) \|x\|
$$
for all $x \in H^2_w$. This is due to the slope condition
$$
(1+\alpha)^{-2}\le w_{n+1}/w_n \le (1+\alpha)^{2},
$$
for all $n \ge 0$.

In the next section, we will prove  that it is possible to choose the numbers $N_k$ so that statements (2)--(4) of Theorem \ref{t0.1} are true. Since statement (3) follows from (with estimates) (4) (see \cite{KwonTreil}, \cite{Nik-shift}), it suffices to 
check only (2) and (4).

\section{Main estimates}
In fact, we will prove even more than what is stated in Theorem \ref{t0.4}. Namely we will show that given $\e>0$, one can pick the numbers $N_k$'s in such a way that 

\begin{enumerate}
\item $\displaystyle{(1+\e)^{-1}\le |k^1_\la(\la)/k_\la^2(\la)| \le 1+\e}$  for all $\la\in\D$, where $k^1_{\la}$ and $k^2_{\la}$ are the reproducing kernels for the Hardy space $H^2$ and for the weighted Hardy space $H^2_w$, respectively.%
%
\footnote{%
To shorten the notation we will use the symbol $k^2_\la$ instead of $k^w_\la$ for the reproducing kernel of $H^2_w$.
}%
%
This means that condition (2) of Theorem \ref{t0.1} holds with $c=1+\e$ (the bundle map between the eigenvector bundles is $k^1_{\bar \la} \mapsto k^2_{\bar\la}$, $\la\in\D$). 

\item
The difference of the curvatures satisfies the estimate 
$$
|\kappa_T(z)-\kappa_{S^*}(z)| \leq \frac{\epsilon}{(1-|z|)^2},
$$
and the measure
$$
d\mu=|\kappa_T(z)-\kappa_{S^*}(z)|(1-|z|)dxdy
$$
is Carleson with Carleson norm $ \|\mu\|_{Carl} \leq \epsilon$. 

\end{enumerate}

Let us first recall the definition of a Carleson measure. Denote by $I \subset \T$ an arc of the unit circle with length $|I|$. Then a positive measure $\mu$ in the closed unit disk is called a \emph{Carleson measure} if 
$$
\mu (Q_I) \leq c \frac{|I|}{2\pi},
$$ 
for some constant $c > 0$, where
$$
Q_I:=\left\{z \in \C: \frac{z}{|z|} \in I, 1-|z| \leq \frac{|I|}{2\pi} \right\}
$$
is the \emph{Carleson window} (\cite{Nik-book-v1} Chap. 6.3). The infimum of such $c>0$ is called the \emph{Carleson norm}, denoted $\|\mu\|_{Carl}$.

The maps $z\mapsto k^1_{\bar z}$ and $z\mapsto k^2_{\bar z}$ are  sections of the eigenvector bundles of the operators $S^*$ and $T$ respectively. 
 Since $\|k^i_{\bar z}\|^2 = (k^i_{\bar z}, k^i_{\bar z}) = k^i_{\bar z}(\bar z)= k^i_z(z)$, we get, recalling formula for the curvature (see Remark \ref{r1.2}) that the corresponding curvatures equal $-\Delta \ln k^i_z(z)$, $i=1, 2$. 

We know that $k^1_z(\la) = 1/(1-\overline z \la)$,  
so $k_{z}^1(z)=\frac{1}{1-|z|^2}$. On the other hand, \eqref{1.1} implies that
$$
k_{z}^2(z)=\sum_{n  \geq 0} \frac{1}{w_n} |z|^{2n}
=\frac{1}{1-|z|^2}+\sum_{k=1}^{\infty} g_k(z),
$$
where
$$
g_k(z)=|z|^{2{N_k}} \left( \sum_{j=1}^{k} c_j|z|^{2j}+ \sum_{j=1}^{k-1}  c_j |z|^{2(2k-j)}\right),
$$
where
$$
c_j=\frac{1}{(1+\epsilon)^{2j}}-1.
$$

Thus,
$$
|\kappa_{S^*}-\kappa_T|=\left|\Delta \ln \frac{k_{z}^2(z)}{k_{z
}^1(z)}\right |=\left|\Delta \ln \left(1+\sum_{k=1}^{\infty} g_k(z)(1-|z|^2)\right) \right|.
$$
Since 
$$
|\Delta \ln f|= \frac{|f \Delta f-|\p f|^2|}{|f|^2},
$$
to prove that conditions (1) and (2) hold,  we need to show that given $\delta > 0$, the $N_k$'s can be chosen so that the following conditions hold for $f(z)=\frac{k_{z}^2(z)}{k_{z}^1(z)}=1+\sum_{k=1}^{\infty} g_k(z)(1-|z|^2)$:

\begin{enumerate}

\item
$1-\delta \leq |f| \leq 1+\delta$;

\item
\begin{enumerate}

\item
$|\Delta f| \leq \frac{\delta}{(1-|z|)^2}$;

\item
$|\p f| \leq \frac{\sqrt{\delta}}{1-|z|}$;
\end{enumerate}

\item
$d\mu=|\Delta f|(1-|z|)dxdy$ is a Carleson measure with $\|\mu\|_{Carl} \leq \delta$;

\item
$d\mu=|\p f|^2(1-|z|)dxdy$ is a Carleson measure with $\|\mu\|_{Carl} \leq \delta$.

\end{enumerate}

\begin{rem}
Since $f$ is a quotient of the reproducing kernels, the above condition (1) implies condition (2) of Theorem \ref{t0.1} with $c=1/(1-\delta)$. 
\end{rem}

But note that these conditions follow from the following four conditions where we consider the functions 
$h_k(z):=g_k(z)(1-|z|^2)$ instead of  $f(z)$:
\begin{enumerate}

\item
$ |h_k| \leq \frac{\delta}{2^k}$;

\item
\begin{enumerate}

\item
$|\Delta h_k| \leq \frac{\delta}{2^k{(1-|z|)^2}}$;
 
 \item
 $|\p h_k| \leq \frac{\delta}{{2^k}(1-|z|)}$;
\end{enumerate}

\item
$d\mu=|\Delta h_k|(1-|z|)dxdy$ is a Carleson measure with $\|\mu\|_{Carl} \leq \frac{\delta}{2^k}$;

\item
$d\mu=|\p h_k|^2(1-|z|)dxdy$ is a Carleson measure with $\|\mu\|_{Carl} \leq \frac{\delta}{2^{2k}}$.\
\end{enumerate}

It is easy to see that the first three conditions for $h_k$ imply those  for $f$ because they are just  linear (in $h_k$) estimates. As for the fourth condition, let us notice that condition (4) for $h_k$ simply means that for every Carleson window $Q_I$,
\[
\|\p h_k\|_{L^2(Q_I, (1-|z|)dxdy)} \leq \frac{\sqrt\delta}{2^{k}}|I|^{\frac{1}{2}} .
\]
Since $\p f = \sum_{k\ge 1} \p h_k$, we conclude that 
\[
\|\p f \|_{L^2(Q_I, (1-|z|)dxdy)} \leq \sum_{k=1}^\infty \frac{\sqrt\delta}{2^{k}}|I|^{1/2} =  \sqrt\delta |I|^{1/2},
\]
which is exactly statement (4) for $f$.

Next, note that since the $h_k$ can be written down as $h_k={z}^{2N_k}(1-|z|^2)\sum_{j=1}^{2k-1}a_j^k|z|^{2j}$,  it suffices to show that for the elementary functions
$$
\psi_N=|z|^{2N}(1-|z|^2)=z^N {\bar{z}}^N-z^{N+1} {\bar{z}}^{N+1},
$$
the analogues of the above conditions (1)--(4) hold, and that by taking sufficiently large $N$ we can make the right sides of the inequalities there as small as we want. Hence Theorem \ref{t0.4} will follow from the Lemma given below:

\begin{lm}
The following statements are true for the function $\psi_N=|z|^{2N}(1-|z|^2)$:

\begin{enumerate}

\item
$\psi_N(z) \longrightarrow 0$  uniformly  on $\D$ as $N \rightarrow \infty$;

\item
$|\Delta \psi_N(z) |(1-|z|)^2 \longrightarrow  0$ uniformly on $\D$ as $N \rightarrow \infty$;

\item
$\left|\frac{\p \psi_N(z)}{\p z}\right|^2(1-|z|)^2 \longrightarrow 0$ uniformly on $\D$ as $N \rightarrow \infty$;

\item
The measure $d\mu= d\mu_N =|\Delta \psi_N(z)|(1-|z|)dxdy$ is a Carleson measure with $\|\mu\|_{Carl} \rightarrow 0$ as $N \rightarrow \infty$;

\item
The measure $d\mu=d\mu_N=\left|\frac{\p \psi_N(z)}{\p z}\right|^2(1-|z|)dxdy$ is a Carelson measure with $\|\mu\|_{Carl} \rightarrow 0$ as $N \rightarrow \infty$.

\end{enumerate}

\end{lm}

\begin{proof}
Direct calculations show that 
\begin{align*}
\frac{\p \psi_N}{\p z} & =Nz^{N-1}{\bar{z}}^N-(N+1)z^{N}{\bar{z}}^{N+1}, \text{ and }
\\
\frac{\p \psi_N}{\p \overline z} & =Nz^{N}{\bar{z}}^{N-1}-(N+1)z^{N+1}{\bar{z}}^{N}.
\end{align*}
We then have that
\begin{align*}
\Delta \psi_N & =N^2|z|^{2(N-1)}-(N+1)^2|z|^{2N}, \text{ and}
\\
\left|\frac{\p \psi_N}{\p z}\right|^2 & =|z|^{2(2N-1)}(N-(N+1)|z|^2)^2. 
\end{align*}

Since the maximum of $\psi_N$ attained at $|z|=\sqrt{{\frac{N}{N+1}}}$ is
${(\frac{N}{N+1})}^N\frac{1}{N+1}$, statement (1) is obvious.  

Let us prove (2). Denoting $r=|z|$, we have
\begin{align*}
|\Delta \psi_N(z) |(1-|z|)^2 & = \left|  N^2r^{2N-2}-(N+1)^2r^{2N} \right|(1-r)^2 
\\   
& = \left | (N+1)^2r^{2N-2}(1-r^2) - 
(2N+1)r^{2N-2} \right|     (1-r)^2
\\
&\leq 2(N+1)^2r^{2N-2}(1-r)^3+(2N+1)r^{2N-2}(1-r)^2.
\end{align*}
We can see that both terms on the right side of the inequality can be estimated by functions of the same type: 
\begin{align*}
(N+1)^2r^{2N-2}(1-r)^3 &\le C \left( n^kr^{kn}(1-r) \right)^3,  \qquad k=2/3, n=N-1, \\
(2N+1)r^{2N-2}(1-r)^2 &\le  C \left( n^kr^{kn}(1-r) \right)^2,   \qquad k =1/2 , n=2(N-1).
\end{align*}
The function   $f(r)= n^kr^{kn}(1-r)$ attains its maximum value 
\[
 \left( \frac{kn}{kn+1} \right)^{kn}\frac{n^k}{kn+1}
\]
at $r=kn/(kn+1)$. Since $k<1$, 
  the maximum value clearly goes to 0 as $n \rightarrow \infty$. 

Estimating the square root of the expression in statement (3), we get  
\begin{align*}
\left|\frac{\p \psi_N(z)}{\p z} \right|  (1-|z|) & =r^{2N-1}|N-(N+1)r^2|(1-r) 
\\
&\leq r^{2N-1}|(N+1)(1-r^2)|(1-r)+r^{2N-1}(1-r)
\\
& \leq 2r^{2N-1}(N+1)(1-r)^2+r^{2N-1}(1-r).
\end{align*}
Since the maxima for the two terms on the right side of the last inequality are 
$$
2(N+1)\left(\frac{2N-1}{2N+1} \right)^{2N-1}\left(\frac{2}{2N+1}\right)^2,
$$
and
$$
\left(1-\frac{1}{2N}\right)^{2N-1}\frac{1}{2N},
$$
respectively,  we see that $\left|\frac{\p \psi_N(z)}{\p z}\right|(1-|z|)$ approaches $0$ as $N \rightarrow \infty$.

For statement (4), let  $S_I=\{z \in \D: \frac{z}{|z|} \in I \}$ be the smallest  sector of $\D$ with vertex at $0$ containing $Q_I \cap \D$.
Since 
\begin{align*}
\frac{1}{|I|}\mu (Q_I )  & \le   \frac{1}{|I|}\mu (S_I)
\\
& =  \frac{1}{|I|} \int_{S_I}\left( N^2|z|^{2(N-1)}-(N+1)^2|z|^{2N}\right) dxdy,
\end{align*}
we get the following by integrating in polar coordinates:
\begin{align*}
\frac{1}{|I|}\mu(Q_I)&  \le \frac{1}{|I|} |I| \int_0^1 \left( N^2 r^{2(N-1)}-(N+1)^2 r^{2N}\right) r dr 
\\
& = \int^1 _0 |N^2- (N+1)^2r^2|r^{2N-1}(1-r) dr 
\\ 
& \leq 2(N+1)^2 \int^1_0 r^{2N-1}(1-r)^2dr+(2N+1)\int^1_0 r^{2N-1}(1-r)dr
\\ 
& =2(N+1)^2\left(\frac{1}{2N}-\frac{2}{2N+1} +\frac{1}{2N+2}\right)
\\ & \qquad \qquad \qquad \qquad \qquad \qquad \qquad +(2N+1)\left(\frac{1}{2N}-\frac{1}{2N+1}\right)
\\  
& =2(N+1)^2\text{O}\left(\frac{1}{N^3}\right)+2(N+1)\text{O}\left(\frac{1}{N^2}\right).
\end{align*}

Finally, as with statement (4),  replacing $Q_I$ with the sector $S_I$ and integrating in polar coordinates, we reduce statement (5) to the following estimate:
\begin{align*}
\int^1_0 &r^{4N-1}(N  -(N+1)r^2)^2(1-r)dr 
\\ &
\leq 8\int^1_0 r^{4N-1}(N+1)^2(1-r)^3dr+2\int^1_0 r^{4N-1}(1-r)dr
\\ &
=8(N+1)^2(\frac{1}{4N}-\frac{3}{4N+1}+\frac{3}{4N+2}-\frac{1}{4N+3})+2(\frac{1}{4N}-\frac{1}{4N+1}). 
\\ &
=8(N+1)^2\text{O}\Bigl(\frac{1}{N^4}\Bigr)+2\text{O}\left(\frac{1}{N^2}\right).
\end{align*}
The lemma follows by letting $N \rightarrow \infty$.
\end{proof}

\def\cprime{$'$}
\providecommand{\bysame}{\leavevmode\hbox
to3em{\hrulefill}\thinspace}

\end{document}